\documentclass{article}
\usepackage[left=1in,top=1in,right=1in,bottom=1in,letterpaper]{geometry}
\usepackage{amsmath,amssymb,amsfonts}
\usepackage{amsthm}
\usepackage{cite}
\usepackage{makeidx}
\usepackage{lineno}
\usepackage{float}
\restylefloat{table}
\numberwithin{equation}{section}

\usepackage[ruled,vlined,linesnumbered]{algorithm2e}
\usepackage{subfigure}
\usepackage[usenames]{color}

\makeindex
\newif\ifpdf\ifx\pdfoutput\undefined\pdffalse\else\pdfoutput=1\pdftrue\fi

\usepackage[colorlinks=true, pdfstartview=FitV, linkcolor=blue,
                                    citecolor=blue, urlcolor=blue]{hyperref}

  \ifpdf        \usepackage[pdftex]{graphicx}
        \usepackage{epstopdf}
        \usepackage{url}
  \else\usepackage{graphicx}\fi

\newcount\refnum\refnum=0
\def\myref{{\global\advance\refnum by 1} {\bf \large Lecture \the \refnum. }}

\usepackage[normalem]{ulem} % to use \sout

\newcommand{\cB}{{\mathcal{B}}}

\newcommand{\cG}{{\mathcal{G}}}

\newcommand{\cK}{{\mathcal{K}}}
\newcommand{\cL}{{\mathcal{L}}}

\newcommand{\cP}{{\mathcal{P}}}

\newcommand{\cR}{{\mathcal{R}}}

\newcommand{\cY}{{\mathcal{Y}}}
\newcommand{\cZ}{{\mathcal{Z}}}

\newcommand{\RR}{\mathbb{R}} % real
\renewcommand{\SS}{{\mathbb{S}}} % symmetric matrix
\newcommand{\st}{\mbox{ s.t. }}

\DeclareMathOperator*{\argmin}{arg\,min} % argmin
\newcommand{\bc}{\begin{center}}
\newcommand{\ec}{\end{center}}

\newcommand{\bdm}{\begin{displaymath}}
\newcommand{\edm}{\end{displaymath}}

\newcommand{\beq}{\begin{equation}}
\newcommand{\eeq}{\end{equation}}

\newcommand{\bfl}{\begin{flushleft}}
\newcommand{\efl}{\end{flushleft}}

\newcommand{\bt}{\begin{tabbing}}
\newcommand{\et}{\end{tabbing}}

\newcommand{\beqn}{\begin{eqnarray}}
\newcommand{\eeqn}{\end{eqnarray}}

\newcommand{\beqs}{\begin{align*}} % no equation numbers
\newcommand{\eeqs}{\end{align*}}  % no equation numbers

\newtheorem{theorem}{Theorem}[section]
\newtheorem{definition}{Definition}[section]
\newtheorem{lemma}{Lemma}[section]

\newtheorem{remark}{Remark}[section]

\def\R{\mathbb{R}}

\def\S{\mathbb{S}}
\def\Sn{\S^n}
\def\Snn{\S^{n^2+1}}
\def\Snp{\Sn_+}
\def\Snr{\S^{rn}_+}

\newcommand{\ADMM}{\textbf{ADMM}\,}
\newcommand{\ADMMp}{\textbf{ADMM}}
\newcommand{\QAP}{\textbf{QAP}\,}
\newcommand{\QAPp}{\textbf{QAP}}

\newcommand{\SDP}{\textbf{SDP}\,}
\newcommand{\SDPp}{\textbf{SDP}}

\newcommand{\textdef}[1]{\textit{#1}\index{#1}}

\DeclareMathOperator{\trace}{{trace}}
\DeclareMathOperator{\kvec}{{vec}}
\DeclareMathOperator{\diag}{{diag}}
\title{
	\ADMM for the \SDP relaxation of the \QAP\thanks{This work is
	partially supported by NSERC and AFOSR.}
}

\author{
	\href{https://uwaterloo.ca/combinatorics-and-optimization/about/people/deolivei}
  {Danilo Elias Oliveira}
  \thanks{
        Dept. of Combinatorics and Optimization, University of Waterloo.
  }
\and
        \href{http://orion.math.uwaterloo.ca/~hwolkowi/}
{Henry Wolkowicz}
        \thanks{Dept. of Combinatorics and Optimization, University of Waterloo.
		Research supported by The Natural Sciences and Engineering
                Research Council of Canada
                and by AFOSR.
                Email: {hwolkowicz@uwaterloo.ca}}
\and
\href{http://www.ima.umn.edu/~yangyang/}
  {Yangyang Xu}
\thanks{Institute for Mathematics and its Applications (IMA),
University of Minnesota}
}

\date{\today}

\begin{document}
\maketitle

\begin{abstract}
	The semidefinite programming \SDP relaxation has
	proven to be extremely strong for many hard discrete
	optimization problems. This is in particular true for the
	quadratic assignment problem \QAPp, arguably one of the hardest NP-hard discrete optimization problems. There are
	several difficulties that arise in efficiently  solving the
	\SDP relaxation, e.g.,~increased dimension; inefficiency of the
	current primal-dual interior point solvers in terms of both time and
	accuracy; and difficulty and high expense in adding cutting plane 
	constraints.

       We propose using the alternating direction method of 
      multipliers \ADMM  to solve the \SDP relaxation. This first order approach allows for inexpensive iterations, a method of
      cheaply obtaining low rank solutions, as well a
      trivial way of adding cutting plane inequalities.
      When compared to current approaches and current best available
      bounds we obtain remarkable robustness, efficiency and improved
      bounds.
\end{abstract}


{\bf Classification code:}
90C22, 90B80, 90C46, 90-08

\tableofcontents
\listoftables
\section{Introduction}
The \textdef{quadratic assignment problem} (\QAPp), in the trace formulation is
\index{\QAPp, quadratic assignment problem}
\begin{equation}\label{qap}
	\textdef{$p_X^*$}:=\min_{X\in\Pi_n} \langle AXB-2C, X\rangle,
\end{equation}
\index{optimal value \QAPp, $p^*_X$}
where $A,B\in \Sn$ are real symmetric $n\times n$ matrices, $C$ is a real
$n\times n$ matrix, $\langle\cdot\,,\cdot\rangle$ denotes the \textdef{trace
inner product, $\langle Y,X\rangle=\trace YX^\top$},
and $\Pi_n$ denotes the set of $n\times n$
permutation matrices. A typical
objective of the \QAP is to assign $n$ facilities to $n$ locations
while minimizing total cost. The assignment cost is the sum of costs
using the flows in $A_{ij}$ between a pair of facilities $i,j$
multiplied by the
distance in $B_{st}$ between their assigned locations $s,t$
and adding on the
location costs of a facility $i$ in a position $s$ given in $C_{is}$.

It is well known that the \QAP is an NP-hard problem and that problems
with size as moderate as $n=30$ still remain difficult to solve. Solution
techniques rely on calculating efficient lower bounds. An important tool for
finding lower bounds is the work in \cite{KaReWoZh:94} that provides a
\textdef{semidefinite programmming} (\SDPp), relaxation of \eqref{qap}.
The methods of choice for \SDP are based on a
\textdef{primal-dual interior-point, p-d i-p},
approach. These methods cannot solve large problems,
have difficulty in obtaining high accuracy solutions and cannot
properly exploit sparsity. Moreover, it is very expensive to add on
nonnegativity and cutting plane constraints.
The current state for finding bounds and solving \QAP is given in
e.g.,~\cite{MR2546331,MR2004391,AnsBrix:99,MR2166543,rendlsotirov:06}.

\index{\SDPp, semidefinite programmming}

In this paper we study an \textdef{alternating direction method of multipliers}
(\ADMMp), for solving the \SDP   relaxation of the \QAPp.
We compare this with the best known results given in
\cite{rendlsotirov:06} and with the best known bounds found at SDPLIB
\cite{BurKarRen91}. and with
a p-d i-p methods based on the so-called
HKM direction. We see that the \ADMM method is significantly
faster and obtains high accuracy solutions. In addition
there are advantages in obtaining low rank \SDP solutions 
that provide better feasible
approximations for the \QAP for upper bounds. Finally, it is
trivial to add nonnegativity and rounding constraints while iterating so
as to obtain significantly stronger bounds and also maintain sparsity
during the iterations.

We note that previous success for \ADMM for \SDP in presented
in \cite{MR2741485}. A detailed
survey article for \ADMM can be found in
\cite{BoydParikhChuPeleatoEckstein:11}.

\section{A New Derivation for the \SDP Relaxation}
We start the derivation from the following equivalent
quadratically constrained quadratic problem
\begin{align}\label{qc1}
\min_X &\, \langle AXB-2C, X\rangle\cr
\st & X_{ij}X_{ik}=0,\ X_{ji}X_{ki}=0,\,\forall i,\,\forall j\neq k,\cr
& X_{ij}^2-X_{ij}=0,\,\forall i,j,\\
& \sum_{i=1}^n X_{ij}^2-1=0,\,\forall j,\ \sum_{j=1}^n X_{ij}^2-1=0,\,\forall i.\nonumber
\end{align}

\begin{remark}
Note that the quadratic orthogonality
constraints $X^\top X=I,\,XX^\top=I$, and
the linear row and column sum constraints
$Xe=e, \,X^\top e=e$ can all be linearly represented using linear combinations
of those in \eqref{qc1}.

In addition, the first set of constraints,
the elementwise orthogonality of the row and columns of $X$, are
referred to as the \textdef{gangster constraints}. They are particularly
strong constraints and enable many of the other constraints to be
redundant. In fact, after the facial reduction done below, many of these
constraints also become redundant. (See the definition of the index set
$J$ below.)
\end{remark}

The Lagrangian for \eqref{qc1} is
\begin{align*}
\cL_0(X, U, V, W, u, v) = & \langle AXB-2C, X\rangle + \sum_{i=1}^n\sum_{j\neq k}U_{jk}^{(i)}X_{ij}X_{ik}+\sum_{i=1}^n\sum_{j\neq k}V_{jk}^{(i)}X_{ji}X_{ki}+\sum_{i,j}W_{ij}(X_{ij}^2-X_{ij})\\
& +\sum_{j=1}^n u_j\left(\sum_{i=1}^n X_{ij}^2-1\right)+\sum_{i=1}^n v_i\left(\sum_{j=1}^n X_{ij}^2-1\right).
\end{align*}
The dual problem is a maximization of the dual functional $d_0$,
\begin{equation}\label{dual1}
	\max\ d_0(U,V,W,u,v):=\min_X \cL_0(X,U,V,W,u,v).
\end{equation}
To simplify the dual problem, we homogenize the $X$ terms in $\cL_0$ by
multiplying a unit scalar $x_0$ to degree-1 terms and adding the single
constraint $x_0^2=1$ to the Lagrangian. We let
\begin{align*}
\cL_1(X,x_0,U, V, W, w_0,u, v) = & \langle AXB-2x_0C, X\rangle + \sum_{i=1}^n\sum_{j\neq k}U_{jk}^{(i)}X_{ij}X_{ik}+\sum_{i=1}^n\sum_{j\neq k}V_{jk}^{(i)}X_{ji}X_{ki}+\sum_{i,j}W_{ij}(X_{ij}^2-x_0X_{ij})\\
& +\sum_{j=1}^n u_j\left(\sum_{i=1}^n X_{ij}^2-1\right)+\sum_{i=1}^n v_i\left(\sum_{j=1}^n X_{ij}^2-1\right)+w_0(x_0^2-1).
\end{align*}
This homogenization technique is the same as that in \cite{KaReWoZh:94}.
The new dual problem is
\begin{equation}\label{dual2}
	\max\ d_1(U,V,W,w_0,u,v):=\min_{X,x_0} \cL_1(X,x_0,U, V, W, w_0,u, v).
\end{equation}
Note that $d_1\le d_0$. Hence, our relaxation still yields a lower bound
to \eqref{qc1}. In fact, the relaxations give the same lower bound. This
follows from strong duality of the trust region subproblem as shown in
\cite{KaReWoZh:94}.
Let $x=\kvec (X)$, $y=[x_0; x]$, and $w=\kvec (W)$, where $x,w$ is the
vectorization, columnwise, of $X$ and $W$, respectively. Then
\index{$\kvec$}
\begin{align*}
\cL_1(X, x_0,U, V, W, w_0, u, v) = y^\top\left[L_Q+\cB_1(U)+\cB_2(V)+\text{Arrow}(w,w_0)+\cK_1(u)+\cK_2(v)\right]y-e^\top(u+v)-w_0,
\end{align*}
where 
\begin{align*}
&\cK_1(u)=\text{blkdiag}(0, u\otimes I),\quad \cK_2(v)=\text{blkdiag}(0,I\otimes v),\\
&\text{Arrow}(w,w_0)=\left[\begin{array}{cc}w_0 & - \frac{1}{2}w^\top\\
- \frac{1}{2}w & \text{Diag}(w)\end{array}\right]
\end{align*} 
and 
$$\cB_1(U)=\text{blkdiag}(0, \tilde{U}),\quad\cB_2(V)=\text{blkdiag}(0,\tilde{V}).$$ Here, $\tilde{U}$ and $\tilde{V}$ are $n\times n$ block matrices. $\tilde{U}$ has zero diagonal blocks and the $(j,k)$-th off-diagonal block to be the diagonal matrix $\text{Diag}(U_{jk}^{(1)},\ldots,U_{jk}^{(n)})$ for all $j\neq k$, and $\tilde{V}$ has zero off-diagonal blocks and the $i$-th diagonal block to be $\left[\begin{array}{cccc}0 & V_{12}^{(i)}  & \cdots & V_{1n}^{(i)}\\
V_{21}^{(i)} & 0  & \cdots & V_{2n}^{(i)}\\
\vdots & \vdots & \ddots & \vdots \\
V_{n1}^{(i)} & V_{n2}^{(i)} & \cdots  & 0\end{array}\right]$. Hence, the dual problem \eqref{dual2} is
\begin{align}\label{dqc1}
\max\ & -e^\top(u+v)-w_0\\
\st & L_Q+\cB_1(U)+\cB_2(V)+\text{Arrow}(w,w_0)+\cK_1(u)+\cK_2(v)\succeq 0.\nonumber
\end{align}
Taking the dual of \eqref{dqc1}, we have the \SDP relaxation of \eqref{qc1}:
\begin{align}\label{nrqap}
\min\ & \langle L_Q, Y\rangle\cr
\st & \cG_J(Y) = E_{00},\ \diag(\bar{Y})=y_0,\\
& \text{trace}(\tilde{Y}_{ii})=1,\,\forall i,\ \sum_{i=1}^n \tilde{Y}_{ii}=I,\cr
&Y\succeq 0,\nonumber
\end{align}
where $\tilde{Y}_{ij}$ is an $n\times n$ matrix for each $(i,j)$, and we have assumed the block structure
\begin{equation}
	\label{eq:Ybar}
	Y=\left[\begin{array}{cc}y_{00} & y_0^\top\\
y_0 & \bar{Y}\end{array}\right]; \qquad
\bar Y \text{ made of
$n\times n$ block matrices } \tilde{Y}=(\tilde{Y}_{ij}).
\end{equation}
The index set $J$ and the gangster operator $\cG_J$ are defined
properly below in Definition \ref{def:J}.
(By abuse of notation this is done after the facial reduction which
results in a smaller $J$.)

\begin{remark}
If one more feasible quadratic constraint $q(X)$ can be added to \eqref{qc1} and $q(X)$ cannot be linearly represented by those in \eqref{qc1}, the relaxation following the same derivation as above can be tighter. We conjecture that no more such $q(X)$ exists, and thus \eqref{nrqap} is the tightest among all Lagrange dual relaxation from a quadratically constrained program like \eqref{qc1}.
However, this does not mean that more linear inequality constraints
cannot be added, i.e.,~\textdef{linear cuts}.
\end{remark}

\begin{theorem}[\cite{KaReWoZh:94}]
The matrix $Y$ is feasible for \eqref{nrqap} \emph{if,
and only if}, it is feasible for \eqref{qap2}.
\qed
\end{theorem}

As above, let $x=\kvec X\in \R^{n^2}$ be the vectorization of $X$ by column.
$Y$ is the original matrix variable of the \SDP relaxation
before the facial reduction. It can be motivated from the
\textdef{lifting}
$Y= \begin{pmatrix} 1 \cr \kvec X \end{pmatrix}
\begin{pmatrix} 1 \cr \kvec X \end{pmatrix}^\top$. %, where $\kvec X$ is
%the vectorization of $X$ by column.

The \SDP relaxation of \QAP presented in \cite{KaReWoZh:94} uses
\textdef{facial reduction} to guarantee strict feasibility. The \SDP
obtained is
\begin{equation}
\label{qap1}
\begin{array}{rl}
	\textdef{$p_R^*$}:=\min_R &  \langle  L_Q, \hat{V}R\hat{V}^\top\rangle \\
	\st &  \cG_J(\hat{V}R\hat{V}^\top)=E_{00}\\
	    &   R\succeq 0,
\end{array}
\end{equation}
\index{optimal value primal \SDP relaxation, $p^*_R$}
where the so-called \textdef{gangster operator, $\cG_J$},
fixes all elements indexed by $J$ and zeroes out all others,
\index{$\cG_J$, gangster operator}
\begin{equation}
	\label{eq:LVhat}
L_Q=\left[\begin{array}{cc}0 & -\mbox{vec}(C)^\top\\ -\mbox{vec}(C) &
B\otimes A \end{array}\right], \qquad
\hat{V}=\left[\begin{array}{cc}1 & 0\\ \frac{1}{n}e & V\otimes V\end{array}\right]
\end{equation}
\index{$e$, ones vector}
with $e$ being the vector of all \textdef{ones},
of appropriate dimension
and $V\in\RR^{n\times(n-1)}$ being a basis matrix of the orthogonal
complement of $e$, e.g.,~$V=\begin{bmatrix} I_{n-1}\cr\hline  -e \end{bmatrix}$.
We let $Y=\hat{V}R\hat{V}^\top\in \Snn$.

\begin{lemma}[\cite{KaReWoZh:94}]
	\label{lem:strprimalfeas}
	The matrix \textdef{$\hat R$} defined by
	\[
		\hat R:=
		\left[\begin{array}{c|c}
			1 & 0 \cr
				\hline
			0 & \frac 1{n^2(n-1)}
			        \left(nI_{n-1}-E_{n-1}\right) \otimes
			        \left(nI_{n-1}-E_{n-1}\right)
		\end{array}\right] \in \S^{(n-1)^2+1}_{++}
	\]
	is (strictly) feasible for \eqref{qap1}.
	\index{strictly feasible pair, $(\hat R,\hat Y,\hat Z)$}
	\qed
\end{lemma}
\begin{definition}
	\label{def:J}
The gangster operator $\cG_J:\Snn \rightarrow \Snn$ and is defined by
\[
	\cG_J(Y)_{ij}=\left\{
		\begin{array}{cc}
			Y_{ij} & \text{ if } (i,j)\in J \text{ or } (j,i)\in J\\
			0  & \text{otherwise}
		\end{array}
		\right.
\]
By abuse of notation, we let the same symbol denote the projection onto
$\R^{|J|}$. We get the two equivalent primal constraints:
\[
	 \cG_J(\hat{V}R\hat{V}^\top)=E_{00}\in \Snn; \qquad \qquad
	 \cG_J(\hat{V}R\hat{V}^\top)=\cG_J(E_{00})\in \R^{|J|}.
\]
Therefore, the dual variable for the first form is $Y\in \Snn$. However,
the dual variable for the second form is $y\in  \R^{|J|}$ with the
adjoint now yielding $Y=\cG^*_J(y) \in \Snn$ obtained by symmetrization
and filling in the missing elements with zeros.

The \textdef{gangster index set, $J$} is defined to be ${(00)}$
union the set of of indices
$i<j$ in the matrix $\bar Y$ in \eqref{eq:Ybar} corresponding to:
\begin{enumerate}
	\item
the off-diagonal elements in the $n$ diagonal blocks;
	\item
the diagonal elements in the off-diagonal blocks except for the last
column of off-diagonal blocks and also not the $(n-2),(n-1)$ off-diagonal block.
(These latter off-diagonal block constraints are redundant after the facial
reduction.)
\end{enumerate}
\index{$J$, gangster index set}
\end{definition}

We note that the gangster operator is self-adjoint, $\cG_J^*=\cG_J$.
Therefore, the dual of \eqref{qap1} can be written as the following.
\begin{equation}
\label{qap1D}
\begin{array}{rll}
	\textdef{$d_Y^*$}:=\max\limits_Y &  \langle  E_{00}, Y\rangle &\quad (=Y_{00}) \\
\st &  \hat{V}^\top \cG_J(Y)\hat{V} \preceq \hat{V}^\top L_Q \hat{V}\\
\end{array}
\end{equation}
\index{optimal value dual \SDP relaxation, $d^*_Y$}
Again by abuse of notation, using the same symbol twice,
we get the two equivalent dual constraints:
\[
 \hat{V}^\top \cG_J(Y)\hat{V} \preceq \hat{V}^\top L_Q \hat{V};
	  \qquad \qquad
 \hat{V}^\top \cG^*_J(y)\hat{V} \preceq \hat{V}^\top L_Q \hat{V}.
\]
As above, the dual variable for the first form is $Y\in \Snn$ and
for the second form is $y\in  \R^{|J|}$. We have used $\cG^*$ for the
second form to emphasize that only the first form is self-adjoint.

\begin{lemma}[\cite{KaReWoZh:94}]
	\index{strictly feasible pair, $(\hat R,\hat Y,\hat Z)$}
	\label{lem:strdualfeas}
	The matrices \textdef{$\hat Y$},\textdef{$\hat Z$}, with
	$M>0$ sufficiently large, defined by
	\index{strictly feasible pair, $(\hat R,\hat Y,\hat Z)$}
	\[
		\hat Y:=
		M\left[\begin{array}{c|c}
			n & 0 \cr
				\hline
				0 & I_n\otimes (I_n-E_n)
		\end{array}\right]\in \S^{(n-1)^2+1}_{++}, \quad
\hat Z:= \hat{V}^\top L_Q \hat{V}- \hat{V}^\top \cG_J(\hat Y)\hat{V}
		\in \S^{(n-1)^2+1}_{++}.
	\]
	and are (strictly) feasible for \eqref{qap1D}.
	\qed
\end{lemma}

\section{A New \ADMM Algorithm for the \SDP Relaxation}
We can write \eqref{qap1} equivalently as
\begin{equation}\label{qap2}
\min_{R, Y}\, \langle L_Q, Y\rangle, \st \cG_J(Y)=E_{00},\, Y=\hat{V}R\hat{V}^\top,\, R\succeq 0.
\end{equation}
The \textdef{augmented Lagrange} of \eqref{qap2} is
\begin{equation}
	\label{eq:augmlagr}
	\cL_A(R, Y, Z)=\langle L_Q, Y\rangle+ \langle Z, Y-\hat{V}R\hat{V}^\top\rangle + \frac{\beta}{2}\|Y-\hat{V}R\hat{V}^\top\|_F^2.
\end{equation}
Recall that $(R,Y,Z)$ are the primal reduced, primal, and dual variables
respectively. We denote $(R,Y,Z)$ as the \emph{current iterate}.
We let \textdef{$\Snr$} denote the matrices in $\Snp$ with rank at
most $r$.
Our new algorithm is an application of the \textdef{alternating direction
method of multipliers} \ADMMp, that uses the augmented Lagrangian
in \eqref{eq:augmlagr} and performs the following updates for
$(R_+,Y_+,Z_+)$:
\begin{subequations}\label{admm1}
\begin{align}
R_+=& \argmin_{R\in \Snr} \cL_A(R, Y, Z),\label{update-r}\\
Y_+=& \argmin_{Y\in \cP_i} \cL_A(R_+, Y, Z),\label{update-y}\\
Z_+=&\, Z+\gamma\cdot\beta(Y_+-\hat{V}R_+\hat{V}^\top)\label{update-z},
\end{align}
\end{subequations}
where the simplest case for the polyhedral constraints
$\cP_i$ is the linear manifold from the \textdef{gangster constraints}:
\[
	\textdef{$\cP_1=\{Y\in \Snn: \cG_J(Y)=E_{00}\}$}
\]
We use this notation as we add additional
simple polyhedral constraints. The second case is the polytope:
\[
\textdef{$\cP_2=\cP_1\cap \{0\leq Y\leq 1\}$}.
\]

Let $\hat{V}$ be normalized such that $\hat{V}^\top\hat{V}=I$. Then if $r=n$, the $R$-subproblem can be explicitly solved by
\begin{equation}
	\label{eq:Rproj}
\begin{array}{rcl}
	R_+&=&\argmin_{R\succeq 0}\langle Z, Y-\hat{V}R\hat{V}^\top\rangle + \frac{\beta}{2}\|Y-\hat{V}R\hat{V}^\top\|_F^2\\
    &=&\argmin_{R\succeq
0}\left\|Y-\hat{V}R\hat{V}^\top+\frac{1}{\beta}Z\right\|_F^2\\
    &=&\argmin_{R\succeq 0}\left\|R-\hat{V}^\top\big(Y+\frac{1}{\beta}Z\big)\hat{V}\right\|_F^2\\
 &=&\cP_{\SS_+}\left(\hat{V}^\top\big(Y+\frac{1}{\beta}Z\big)\hat{V}\right),
\end{array}
\end{equation}
where $\SS_+$ denotes the \SDP cone, and $\cP_{\SS_+}$ is the projection to $\SS_+$. For any symmetric matrix $W$, we have
$$\cP_{\SS_+}(W)=U_+\Sigma_+U_+^\top,$$
 where $(U_+,\Sigma_+)$ contains the positive eigenpairs of $W$ and $(U_-,\Sigma_-)$ the negative eigenpairs.

If $i=1$ in \eqref{update-y}, the $Y$-subproblem also has closed-form solution:
\begin{align}
Y_+=&\argmin_{\cG_J(Y)=E_{00}} \langle L_Q, Y\rangle+ \langle Z, Y-\hat{V}R_+\hat{V}^\top\rangle + \frac{\beta}{2}\|Y-\hat{V}R_+\hat{V}^\top\|_F^2\cr
=& \argmin_{\cG_J(Y)=E_{00}}\left\|Y-\hat{V}R_+\hat{V}^\top+\frac{L_Q+Z}{\beta}\right\|_F^2\cr
=& E_{00}+\cG_{J^c}\left(\hat{V}R_+\hat{V}^\top-\frac{L_Q+Z}{\beta}\right)\label{y-update1}
\end{align}

The advantage of using \ADMM is that its complexity only slightly
increases while we add more constraints to \eqref{qap1} to tighten the
\SDP relaxation. If $0\le \hat{V}R\hat{V}^\top \le 1$ is added in \eqref{qap1}, then we have constraint $0\le Y\le 1$ in \eqref{qap2} and reach to the problem
\begin{equation}\label{qap3}
	\textdef{$p_{RY}^*$}:=\min_{R, Y}\, \langle L_Q, Y\rangle, \st \cG_J(Y)=E_{00},\, 0\le Y\le 1,\, Y=\hat{V}R\hat{V}^\top,\, R\succeq 0.
\end{equation}
\index{optimal value \ADMM relaxation, $p^*_{RY}$}
The \ADMM for solving \eqref{qap3} has the same $R$-update and $Z$-update as those in \eqref{admm1}, and the $Y$-update is changed to
\begin{equation}\label{y-update2}
Y_+=E_{00}+\min\left(1,\,\max\left(0,\,\cG_{J^c}\big(\hat{V}R_+\hat{V}^\top-\frac{L_Q+Z}{\beta}\big)\right)\right).
\end{equation}
With nonnegativity constraint, the less-than-one constraint is redundant but makes the algorithm converge faster.

\subsection{Lower bound}
If we solve \eqref{qap1} or \eqref{qap2} exactly or to a very high accuracy, 
we get a lower bound of the original \QAPp. However, the problem size of 
\eqref{qap1} or \eqref{qap2} can be  extremely large, and thus having an 
exact or highly accurate solution may take extremely long time. In the 
following, we provide an inexpensive way to get a lower bound from the output of 
our algorithm that solves \eqref{qap2} to a moderate accuracy. Let 
$(R^{out},Y^{out},Z^{out})$ be the output of the \ADMM for \eqref{qap3}. 
\begin{lemma}
Let
\[
	\textdef{$\cR:=\{R\succeq 0\}$},\quad 
	\textdef{$\cY:=\{Y: \cG_J(Y)=E_{00},\,0\le Y\le 1\}$}, \quad
			 \textdef{$\cZ:=\{Z: \hat{V}^\top
			Z\hat{V}\preceq 0\}$}. 
\]
Define the \ADMM \textdef{dual function}
\[
	\textdef{$g(Z)$}:= \min_{Y\in\cY} \{\langle L_Q+Z, Y\rangle\}.
\]
Then the \textdef{dual problem of \ADMM \eqref{qap3}} is defined as
follows and satisfies weak duality.
\[
	\begin{array}{rcl}
		\textdef{$d^*_Z$}
		&:=& 
	\max\limits_{Z\in\cZ}  g(Z)
	\\&\leq &p_{R}^*. 
\end{array}
\]
\end{lemma}
\begin{proof}
The dual problem of \eqref{qap3} can be derived as
\begin{align*}
d^*_Z:=&\max_{Z}\min_{R\in\cR,Y\in\cY} \langle L_Q, Y\rangle + \langle Z,  Y-\hat{V}R\hat{V}^\top\rangle\\
=&\max_{Z}\min_{Y\in\cY} \langle L_Q, Y\rangle 
	+ \langle Z,  Y\rangle
	+  \min_{R\in\cR}\langle Z,  -\hat{V}R\hat{V}^\top\rangle\\
=&\max_{Z}\min_{Y\in\cY} \langle L_Q, Y\rangle 
	+ \langle Z,  Y\rangle
	+  \min_{R\in\cR}\langle \hat{V}^\top Z\hat{V},  -R\rangle\\
=& \max_{Z\in\cZ} \min_{Y\in\cY} \langle L_Q+Z, Y\rangle,\\
=& \max_{Z\in\cZ} g(Z)
\end{align*}
Weak duality follows in the usual way by exchanging the max and min.
\end{proof}
For any $Z\in\cZ$, we have $g(Z)$ is a lower bound of \eqref{qap3} 
and thus of the original \QAPp. We use the dual function value of the projection
\textdef{$g\big(\cP_\cZ(Z^{out})\big)$}
as the lower bound, and next we show how to get $\cP_\cZ(\tilde{Z})$ for any symmetric matrix $\tilde{Z}$.

Let $\hat{V}_{\perp}$ be the orthonormal basis of the null space of $\hat{V}$. Then $\bar{V}=(\hat{V},\hat{V}_\perp)$ is an orthogonal matrix. Let $\bar{V}^\top Z \bar{V}=W=\left[\begin{array}{cc}W_{11}& W_{12}\\ W_{21} & W_{22}\end{array}\right]$, and we have
$$\hat{V}^\top Z\hat{V}\preceq 0\Leftrightarrow \hat{V}^\top Z\hat{V}=\hat{V}^\top \bar{V}W\bar{V}^\top\hat{V}=W_{11}\preceq 0.$$
Hence,
\begin{align*}
\cP_\cZ(\tilde{Z})= & \argmin_{Z\in\cZ}\|Z-\tilde{Z}\|_F^2\\
=& \argmin_{W_{11}\preceq 0}\|\bar{V}W\bar{V}^\top - \tilde{Z}\|_F^2\\
=& \argmin_{W_{11}\preceq 0}\|W-\bar{V}^\top\tilde{Z}\bar{V}\|_F^2\\
=&\left[\begin{array}{cc}\cP_{\SS_-}(\tilde{W}_{11}) & \tilde{W}_{12}\\
\tilde{W}_{21} & \tilde{W}_{22}\end{array}\right],
\end{align*}
where $\SS_-$ denotes the negative semidefinite cone, and we have assumed $\bar{V}^\top\tilde{Z}\bar{V}=\left[\begin{array}{cc}\tilde{W}_{11} & \tilde{W}_{12}\\
\tilde{W}_{21} & \tilde{W}_{22}\end{array}\right]$. Note that $\cP_{\SS_-}(W_{11})=-\cP_{\SS_+}(-W_{11})$.

\subsection{Feasible solution of \QAPp}
Let $(R^{out},Y^{out},Z^{out})$ be the output of the \ADMM for \eqref{qap3}. Assume the largest eigenvalue and the corresponding eigenvector of $Y$ are $\lambda$ and $v$. We let $X^{out}$ be the matrix reshaped from the second through the last elements of the first column of $\lambda v v^\top$. Then we solve the linear program
\begin{equation}\label{linprog}
\max_X \langle X^{out}, X\rangle, \st Xe = e,\, X^\top e = e,\, X\ge 0
\end{equation}
by simplex method that gives a basic optimal solution, i.e., a permutation matrix.

\subsection{Low-rank solution} 
Instead of finding a feasible solution through \eqref{linprog}, we can directly get one by restricting $R$ to a rank-one matrix, i.e., $\mbox{rank}(R)=1$ and $R\in\SS_+$. With this constraint, the $R$-update can be modified to
\begin{equation}
R_+=\cP_{\SS_+\cap\cR_1}\left(\hat{V}^\top\big(Y+\frac{Z}{\beta}\big)\hat{V}\right),
\end{equation}
where $\cR_1=\{R: \mbox{rank}(R)= 1\}$ denotes the set of rank-one matrices. For a symmetric matrix $W$ with largest eigenvalue $\lambda>0$ and corresponding eigenvector $w$, we have
$$\cP_{\SS_+\cap\cR_1}=\lambda ww^\top.$$

%%{\color{blue} \textbf{Question 1:} we have global convergence of the algorithm for \eqref{qap3}. Can we show its convergence rate?
%%
%%\textbf{Question 2:} Can we show the convergence of the algorithm for \eqref{qap3} with additional rank-one constraint?
%%
%%\textbf{Question 3:} How to use warm-start technique to decrease the computational complexity?
%%}
     %We compare with results from ?????  the spectral bundle method
     %tests from \cite{rendlsotirov:06} and also the spectral bundle with
     %second order information \cite{MR3175519}??????.
     %We use the Nugent test set as well as the test sets from the
     
%\input{projectedGNsection}

     \subsection{Different choices for $V,\widehat V$}
The matrix $\widehat V$ is essential in the steps of the algorithm, see
e.g.,~\eqref{eq:Rproj}. A sparse $\widehat V$ helps in the projection if one
is using a sparse eigenvalue code. We have compared several.
One is based on applying a QR algorithm
to the original simple $V$ from the definition of $\hat V$
in \eqref{eq:LVhat}. The other two are based on the approach in 
\cite{HaoWangPongWolk:14} and we present the most successful here.
The orthogonal $V$ we use is
\[
   V =
\begin{bmatrix}
\begin{bmatrix}
\begin{bmatrix}
I_{\left \lfloor \frac n2 \right \rfloor} \otimes
                                      \frac 1{\sqrt 2} \begin{bmatrix}
                                      1\cr -1
                                      \end{bmatrix}
\end{bmatrix}
\cr 0_{(n-2\left \lfloor \frac n2 \right \rfloor),
          \left \lfloor \frac n2 \right \rfloor}
\end{bmatrix}
\begin{bmatrix}
\begin{bmatrix}
I_{\left \lfloor \frac n4 \right \rfloor} \otimes
                                     \frac 12 \begin{bmatrix}
                         1 \cr 1 \cr -1 \cr -1
                                      \end{bmatrix}
\end{bmatrix}
\cr 0_{(n-4\left \lfloor \frac n4 \right \rfloor),
          \left \lfloor \frac n4 \right \rfloor}
\end{bmatrix}
\begin{bmatrix}
\ldots
\end{bmatrix}
\begin{bmatrix}
\widehat V
\end{bmatrix}
\end{bmatrix}_{n\times n-1}
\]
i.e.,~the block matrix consisting of $t$ blocks formed from Kronecker products
along with one block $\widehat V$ to complete the appropriate size so that 
$V^\top V=I_{n-1}$, $V^\top e=0$.
We take advantage of the $0$, $1$ structure of the Kronecker blocks and
delay the scaling for the normalization till the end. 
The main work in the low rank projection part of the algorithm is to evaluate 
one (or a few) eigenvalues of $W=\widehat V^\top( Y+\frac{1}{\beta}Z)\hat{V}$
to obtain the update $R_+$.
\[
 Y+\frac{1}{\beta}Z
	= \begin{bmatrix}
		\rho & w^\top \cr w & \bar W
	\end{bmatrix}.
\]
We let 
\[
	K:=V\otimes V,\quad \alpha=1/\sqrt 2,\quad v=\frac 1{\sqrt 2
	n}e,\quad x=\begin{pmatrix} x_1 \cr \bar x \end{pmatrix}.
\]
The structure for $\widehat V$ in
\eqref{eq:LVhat} means that we can  evaluate the product for $Wx$ as
\[
	\begin{array}{rcl}
\begin{bmatrix}\alpha & 0\\ v & K\end{bmatrix}^\top
	 \begin{bmatrix}
		\rho & w^\top \cr w & \bar W
	\end{bmatrix}
\begin{bmatrix}\alpha & 0\\ v & K\end{bmatrix}x
		&=&
\begin{bmatrix}\alpha & 0\\ v & K\end{bmatrix}^\top
	 \begin{bmatrix}
		\rho & w^\top \cr w & \bar W
	\end{bmatrix}
	 \begin{pmatrix}
		\alpha x_1  \cr x_1v + K\bar x
	\end{pmatrix}
		\\&=&
\begin{bmatrix}\alpha & v^\top \\ 0 & K^\top\end{bmatrix}
	 \begin{pmatrix}
		 \rho	\alpha x_1  +w^\top (x_1v + K\bar x)\cr
		 \alpha x_1w  +\bar W(x_1v + K\bar x)
	\end{pmatrix}
		\\&=&
	 \begin{pmatrix}
		 \rho	\alpha^2 x_1  +\alpha w^\top (x_1v + K\bar x)+
		 v^\top\left(\alpha x_1w  +\bar W(x_1v + K\bar x)\right)
		 \cr
		 K^\top \left( \alpha x_1w  +\bar W(x_1v + K\bar
		 x)\right)
	\end{pmatrix}
		\\&=&
	 \begin{pmatrix}
		 \rho\alpha^2 x_1  
		 +\left(\alpha w^\top +v^\top\bar W\right)        
		 \left(x_1v + K\bar x\right)+
		 v^\top\left(\alpha x_1w  \right)
		 \cr
		 K^\top \left( \alpha x_1w  
			 +\bar W(x_1v + K\bar x)    \right)
	\end{pmatrix}.
\end{array}
\]
We emphasize that $V\otimes V= (\bar V\otimes \bar V)/(D\otimes
D)$, where $\bar V$ denotes the unscaled $V$, $D$ is the diagonal
matrix of scale factors to obtain the orthogonality in $V$, and
$/$ denotes the MATLAB division on the right, multiplication by
the inverse on the right. Therefore, we can evaluate
\[
	K^\top \bar W K= 
	(V\otimes V)^\top 
	\bar W(V\otimes V )
	= (\bar V\otimes \bar V)^\top \left[
	(D\otimes D)\backslash \bar W /(D\otimes D)
\right] (\bar V\otimes \bar V).
\]

\section{Numerical experiments} We illustrate our results in Table
\ref{table:probsall45tol5} on the forty five \QAP
instances I and II, see~\cite{BurKarRen91,MR1457185,rendlsotirov:06}. 
The optimal solutions are in column $1$ and current
best known lower bounds from \cite{rendlsotirov:06} are in column $3$ 
marked \emph{bundle}. The p-d i-p lower bound is given in the
column marked \emph{HKM-FR}. (The code failed to find a lower bound on
several problems marked $-1111$.) These bounds were obtained using the facially
reduced \SDP relaxation and exploiting the low rank (one and two) of the
constraints. We used SDPT3 \cite{MR1778429}.\footnote{We do not include
	the times as they were much greater than for the ADMM approach,
e.g.,~hours instead of minutes and a day instead of an hour.}

Our \ADMM lower bound follows in column $4$. We see that it is at least as
good as the current best known bounds in every instance. The percent
improvement is given in column $7$.
We then present the best upper bounds from our heuristics in column $5$.
This allows us to calculate the percentage gap in column $6$.
The CPU seconds are then given in the last columns $8-9$ for the high
and low rank approaches, respectively. The last two columns are the
ratios of CPU times. Column $10$ is the ratio of CPU times for the $5$
decimal and $12$ decimal tolerance for the high rank approach.  All the
ratios for the low rank approach are approximately $1$ and not included.
The quality of the bounds did not change for these two
tolerances. However, we consider it of interest to show that the higher
tolerance can be obtained.

The last column $11$ is the ratio of CPU times for the $12$ decimal 
tolerance of the
high rank approach in column $8$ with the CPU times for $9$ decimal
tolerance for the HKM approach. We emphasize that the lower bounds for
the HKM approach were significantly weaker.

We used MATLAB version 8.6.0.267246 (R2015b) on a PC Dell Optiplex 9020
64-bit, with $16$ Gig, running Windows 7.

We heuristically set $\gamma=1.618$ and $\beta=\frac{n}{3}$ in \ADMMp.
We used two different tolerances $1e-12,1e-5$. Solving the \SDP to the
higher accuracy did not improve the bounds. However, it is interesting
that the \ADMM approach was able to solve the \SDP relaxations to such
high accuracy, something the p-d i-p approach has great difficulty with.
We provide the CPU times for both accuracies.
Our times are significantly lower than those reported in
\cite{rendlsotirov:06,MR2166543}, e.g.,~from $10$ hours to less than an
hour.

We emphasize that we have \underline{improved bounds} for all the \SDP 
instances and have provably found exact solutions six of the instances
Had12,14,16,18, Rou12, Tai12a.
This is due to the ability to add all the nonnegativity
constraints and rounding numbers to $0,1$ with essentially zero extra
computational cost. In addition, the rounding appears to improve the
upper bounds as well. This was the case for both using tolerance of $12$
or only $5$ decimals in the \ADMM algorithm.

%%	%\tiny{
%%	\begin{table}[H]
%%	\begin{flushleft}
%%	\begin{scriptsize}
%%	\input{../code/TableTemp1}   % appears to be tol12 TableTemp1 
%%	\end{scriptsize}
%%	\end{flushleft}
%%	\caption{
%%		\href{http://anjos.mgi.polymtl.ca/qaplib/}
%%		\QAP Instances I and II. Requested tolerance $1e-12$. (finished on
%%		dec13/15)
%%	}
%%	\label{table:probsall45tol12}
%%	\end{table}
%%%}
	%\tiny{
	\begin{table}[H]
	\begin{center}
	\begin{scriptsize}
	\resizebox{\textwidth}{!}{\begin{tabular}{|c|ccccccccccc|}
\hline
& 1. & 2. & 3. & 4. & 5. & 6. & 7. ADMM & 8 Tol5 & 9 Tol5 & 10 Tol12/5 & 11 HKM \\
& opt & Bundle \cite{rendlsotirov:06} & HKM-FR & ADMM & feas & ADMM & vs Bundle & cpusec & cpusec & cpuratio & cpuratio \\
& value & LowBnd & LowBnd & LowBnd & UpBnd & \%gap & \%Impr LowBnd & HighRk & LowRk & HighRk & Tol 9 \\
\hline
\hline
Esc16a & 68  &  59  &  50  &  64 &  72  &   11.76  &    7.35  &  2.30e+01  &  4.02 &  4.14  &  9.37 \\
Esc16b & 292  &  288  &  276  &  290 &  300  &    3.42  &    0.68  &  3.87e+00  &  4.55 &  2.15  &  8.08 \\
Esc16c & 160  &  142  &  132  &  154 &  188  &   21.25  &    7.50  &  1.09e+01  &  8.09 &  4.53  &  4.88 \\
Esc16d & 16  &  8  &  -12  &  13 &  18  &   31.25  &   31.25  &  2.14e+01  &  3.69 &  4.87  & 10.22 \\
Esc16e & 28  &  23  &  13  &  27 &  32  &   17.86  &   14.29  &  3.02e+01  &  4.29 &  4.80  &  8.79 \\
Esc16g & 26  &  20  &  11  &  25 &  28  &   11.54  &   19.23  &  4.24e+01  &  4.27 &  2.72  &  8.63 \\
Esc16h & 996  &  970  &  909  &  977 &  996  &    1.91  &    0.70  &  4.91e+00  &  3.53 &  2.33  & 10.60 \\
Esc16i & 14  &  9  &  -21  &  12 &  14  &   14.29  &   21.43  &  1.37e+02  &  4.30 &  2.39  &  8.76 \\
Esc16j & 8  &  7  &  -4  &  8 &  14  &   75.00  &   12.50  &  8.95e+01  &  4.80 &  3.83  &  7.93 \\
Had12  & 1652  &  1643  &  1641  &  1652 &  1652  &    0.00  &    0.54  &  1.02e+01  &  1.08 &  1.06  &  5.91 \\
Had14  & 2724  &  2715  &  2709  &  2724 &  2724  &    0.00  &    0.33  &  3.23e+01  &  1.69 &  1.19  & 10.46 \\
Had16  & 3720  &  3699  &  3678  &  3720 &  3720  &    0.00  &    0.56  &  1.75e+02  &  3.15 &  1.04  & 12.51 \\
Had18  & 5358  &  5317  &  5287  &  5358 &  5358  &    0.00  &    0.77  &  4.49e+02  &  6.00 &  2.22  & 13.28 \\
Had20  & 6922  &  6885  &  6848  &  6922 &  6930  &    0.12  &    0.53  &  3.85e+02  & 12.15 &  4.20  & 14.53 \\
Kra30a & 149936  &  136059  &  -1111  &  143576 &  169708  &   17.43  &    5.01  &  5.88e+03  & 149.32 &  2.22  & 1111.11 \\
Kra30b & 91420  &  81156  &  -1111  &  87858 &  105740  &   19.56  &    7.33  &  4.36e+03  & 170.57 &  3.01  & 1111.11 \\
Kra32  & 88700  &  79659  &  -1111  &  85775 &  103790  &   20.31  &    6.90  &  3.57e+03  & 200.26 &  4.28  & 1111.11 \\
Nug12  & 578  &  557  &  530  &  568 &  632  &   11.07  &    1.90  &  2.60e+01  &  1.04 &  6.61  &  5.93 \\
Nug14  & 1014  &  992  &  960  &  1011 &  1022  &    1.08  &    1.87  &  7.15e+01  &  1.87 &  5.06  &  8.43 \\
Nug15  & 1150  &  1122  &  1071  &  1141 &  1306  &   14.35  &    1.65  &  9.10e+01  &  3.31 &  5.90  &  7.79 \\
Nug16a & 1610  &  1570  &  1528  &  1600 &  1610  &    0.62  &    1.86  &  1.81e+02  &  3.06 &  3.28  & 12.24 \\
Nug16b & 1240  &  1188  &  1139  &  1219 &  1356  &   11.05  &    2.50  &  9.35e+01  &  3.19 &  6.23  & 11.83 \\
Nug17  & 1732  &  1669  &  1622  &  1708 &  1756  &    2.77  &    2.25  &  2.31e+02  &  4.34 &  3.63  & 13.13 \\
Nug18  & 1930  &  1852  &  1802  &  1894 &  2160  &   13.78  &    2.18  &  4.16e+02  &  5.47 &  2.43  & 15.23 \\
Nug20  & 2570  &  2451  &  2386  &  2507 &  2784  &   10.78  &    2.18  &  4.76e+02  & 11.56 &  3.75  & 14.35 \\
Nug21  & 2438  &  2323  &  2386  &  2382 &  2706  &   13.29  &    2.42  &  1.41e+03  & 15.32 &  1.68  & 14.95 \\
Nug22  & 3596  &  3440  &  3396  &  3529 &  3940  &   11.43  &    2.47  &  2.07e+03  & 21.82 &  1.39  & 13.90 \\
Nug24  & 3488  &  3310  &  -1111  &  3402 &  3794  &   11.24  &    2.64  &  1.20e+03  & 29.64 &  3.29  & 1111.11 \\
Nug25  & 3744  &  3535  &  -1111  &  3626 &  4060  &   11.59  &    2.43  &  3.12e+03  & 39.23 &  1.65  & 1111.11 \\
Nug27  & 5234  &  4965  &  -1111  &  5130 &  5822  &   13.22  &    3.15  &  5.11e+03  & 78.18 &  1.58  & 1111.11 \\
Nug28  & 5166  &  4901  &  -1111  &  5026 &  5730  &   13.63  &    2.42  &  4.11e+03  & 83.38 &  2.17  & 1111.11 \\
Nug30  & 6124  &  5803  &  -1111  &  5950 &  6676  &   11.85  &    2.40  &  7.36e+03  & 133.38 &  1.76  & 1111.11 \\
Rou12  & 235528  &  223680  &  221161  &  235528 &  235528  &    0.00  &    5.03  &  2.76e+01  &  0.93 &  0.98  &  6.90 \\
Rou15  & 354210  &  333287  &  323235  &  350217 &  367782  &    4.96  &    4.78  &  3.12e+01  &  2.70 &  8.68  &  9.46 \\
Rou20  & 725522  &  663833  &  642856  &  695181 &  765390  &    9.68  &    4.32  &  1.67e+02  & 10.31 & 10.90  & 16.08 \\
Scr12 & 31410  &  29321  &  23973  &  31410 &  38806  &   23.55  &    6.65  &  4.40e+00  &  1.17 &  2.40  &  5.79 \\
Scr15  & 51140  &  48836  &  42204  &  51140 &  58304  &   14.01  &    4.51  &  1.38e+01  &  2.41 &  1.84  & 10.75 \\
Scr20  & 110030  &  94998  &  83302  &  106803 &  138474  &   28.78  &   10.73  &  1.53e+03  &  9.61 &  1.15  & 17.96 \\
Tai12a & 224416  &  222784  &  215637  &  224416 &  224416  &    0.00  &    0.73  &  1.79e+00  &  0.90 &  1.04  &  6.70 \\
Tai15a & 388214  &  364761  &  349586  &  377101 &  412760  &    9.19  &    3.18  &  2.74e+01  &  2.35 & 14.69  & 10.34 \\
Tai17a & 491812  &  451317  &  441294  &  476525 &  546366  &   14.20  &    5.13  &  6.50e+01  &  4.52 &  7.31  & 12.04 \\
Tai20a & 703482  &  637300  &  619092  &  671675 &  750450  &   11.20  &    4.89  &  1.28e+02  & 10.10 & 14.32  & 15.85 \\
Tai25a & 1167256  &  1041337  &  1096657  &  1096657 &  1271696  &   15.00  &    4.74  &  3.09e+02  & 38.48 &  5.58  & 1111.11 \\
Tai30a & 1818146  &  1652186  &  -1111  &  1706871 &  1942086  &   12.94  &    3.01  &  1.25e+03  & 142.55 & 10.51  & 1111.11 \\
Tho30  & 88900  &  77647  &  -1111  &  86838 &  102760  &   17.91  &   10.34  &  2.83e+03  & 164.86 &  4.74  & 1111.11 \\
\hline
\end{tabular}
}
	\end{scriptsize}
	\end{center}
	\caption{
		\href{http://anjos.mgi.polymtl.ca/qaplib/}
		\QAP Instances I and II. Requested tolerance $1e-5$. 
	}
	\label{table:probsall45tol5}
	\end{table}

\section{Concluding Remarks}
In this paper we have shown the efficiency of using the \ADMM approach
in solving the \SDP relaxation of the \QAP problem. In particular, we
have shown that we can obtain high accuracy solutions of the \SDP
relaxation in less significantly less cost than current
approaches. In addition, the \SDP relaxation includes the nonnegativity
constraints at essentially no extra cost. This results in both a fast
solution and improved lower  and upper bounds for the \QAPp.

In a forthcoming study we propose to include this in a branch and bound
framework and implement it in a parallel programming approach, see
e.g.,~\cite{MR2932722}.
In addition, we propose to test the possibility of using \emph{warm
starts} in the branching/bounding process and test it on the larger test
sets such as used in e.g.,~\cite{MR2546331}.

\addcontentsline{toc}{section}{Index}
\label{ind:index}
\printindex

\bibliographystyle{plain}
%\bibliographystyle{unsrt}
%\bibliography{qap}

\def\cprime{$'$} \def\cprime{$'$} \def\cprime{$'$}
  \def\udot#1{\ifmmode\oalign{$#1$\crcr\hidewidth.\hidewidth
  }\else\oalign{#1\crcr\hidewidth.\hidewidth}\fi} \def\cprime{$'$}
  \def\cprime{$'$} \def\cprime{$'$}

\end{document}